\documentclass[12pt]{article}
\usepackage{amsmath,amssymb,amsthm}
\usepackage{hyperref}
\usepackage{tikz}

\def\picRRR{\begin{tikzpicture}[scale=0.7]
\draw[thick] (-1.5,-0.5) rectangle (1.5,4.5);
\draw[ultra thick] (-1.5,1) to (1.5,1);
\draw[ultra thick] (-1.5,2) to (1.5,2);
\draw[ultra thick] (-1.5,3) to (1.5,3);
\draw[ultra thick] (-1.5,4) to (1.5,4);
\node[circle,fill=white,draw,ultra thick, inner sep=1.5pt] at (0,0) {};
\node[ultra thick] at (0,-1) {$C_0$};
\end{tikzpicture}
\hfil
\begin{tikzpicture}[scale=0.7]
\draw[thick] (-1.5,-0.5) rectangle (1.5,4.5);
\draw[ultra thick] (-1.5,1) to[out=0,in=90] (0,0) to[out=90,in=180] (1.5,1);
\draw[ultra thick] (-1.5,2) to (1.5,2);
\draw[ultra thick] (-1.5,3) to (1.5,3);
\draw[ultra thick] (-1.5,4) to (1.5,4);
\node[circle,fill=white,draw,ultra thick, inner sep=1.5pt] at (0,0) {};
\node[ultra thick] at (0,-1) {$C_1$};
\end{tikzpicture}
\hfil
\begin{tikzpicture}[scale=0.7]
\draw[thick] (-1.5,-0.5) rectangle (1.5,4.5);
\draw[ultra thick] (-1.5,1) to[out=0,in=90] (0,0) to[out=90,in=180] (1.5,2);
\draw[ultra thick] (-1.5,2) to[out=0,in=180] (1.5,1);
\draw[ultra thick] (-1.5,3) to (1.5,3);
\draw[ultra thick] (-1.5,4) to (1.5,4);
\node[circle,fill=white,draw,ultra thick, inner sep=1.5pt] at (0,0) {};
\node[ultra thick] at (0,-1) {$C_2$};
\end{tikzpicture}
\hfil
\begin{tikzpicture}[scale=0.7]
\draw[thick] (-1.5,-0.5) rectangle (1.5,4.5);
\draw[ultra thick] (-1.5,1) to[out=0,in=90] (0,0) to[out=90,in=180] (1.5,3);
\draw[ultra thick] (-1.5,2) to[out=0,in=180] (1.5,1);
\draw[ultra thick] (-1.5,3) to[out=0,in=180] (1.5,2);
\draw[ultra thick] (-1.5,4) to (1.5,4);
\node[circle,fill=white,draw,ultra thick, inner sep=1.5pt] at (0,0) {};
\node[ultra thick] at (0,-1) {$C_3$};
\end{tikzpicture}
\hfil
\begin{tikzpicture}[scale=0.7]
\draw[thick] (-1.5,-0.5) rectangle (1.5,4.5);
\draw[ultra thick] (-1.5,1) to[out=0,in=90] (0,0) to[out=90,in=180] (1.5,4);
\draw[ultra thick] (-1.5,2) to[out=0,in=180] (1.5,1);
\draw[ultra thick] (-1.5,3) to[out=0,in=180] (1.5,2);
\draw[ultra thick] (-1.5,4) to[out=0,in=180] (1.5,3);
\node[circle,fill=white,draw,ultra thick, inner sep=1.5pt] at (0,0) {};
\node[ultra thick] at (0,-1) {$C_4$};
\end{tikzpicture}}

\def\picAAA{
\begin{tikzpicture}
\node at (0,0) {$\langle0,1,0,1\rangle$};
\node at (3,0) {$\langle1,0,1,0\rangle$};
\node at (6,0) {$\langle1,1,0,0\rangle$};
\node at (10,0) {$\langle0,0,1,1\rangle$};
\node at (8,1.5) {$\langle1,0,0,1\rangle$};
\node at (8,-1.5) {$\langle0,1,1,0\rangle$};
\draw[->, thick] (1,0.125)--(2,0.125);
\draw[->, thick] (2,-0.125)--(1,-0.125);
\draw[->, thick] (4,0.125)--(5,0.125);
\draw[->, thick] (5,-0.125)--(4,-0.125);
\draw[->,thick,rounded corners] (7,-0.125)--(7.25,-0.125)--(7.25,0.125)--(7,0.125);
\draw[->,thick] (8,1.125)--(8,-1.125);
\draw[->,thick] (6.5,0.375)--(7.5,1.125);
\draw[->,thick] (7.5,-1.125)--(6.5,-0.375);
\draw[->,thick] (8.5,1.125)--(9.5,0.375);
\draw[->,thick] (9.5,-0.375)--(8.5,-1.125);
\draw[->,thick] (7,1.5)--(3.5,0.375);
\draw[->,thick] (3.5,-0.375)--(7,-1.5);
\node at (1.5,0.35) {\small $0$};
\node at (1.5,-0.35) {\small $4$};
\node at (4.5,0.35) {\small $1$};
\node at (4.5,-0.35) {\small $3$};
\node at (7.45,0) {\small $2$};
\node at (8.2,0) {\small $2$};
\node at (6.85,0.95) {\small $4$};
\node at (6.85,-0.95) {\small $0$};
\node at (9.15,0.95) {\small $4$};
\node at (9.15,-0.95) {\small $0$};
\node at (5.2,1.15) {\small $1$};
\node at (5.2,-1.15) {\small $3$};
\end{tikzpicture}}

\def\picQQQ{
\begin{tikzpicture}[scale=.5]
\fill[color=white!50!black] (0,5) circle (0.5);
\draw (1,5) circle (0.5);
\fill[color=white!50!black] (0,4) circle (0.5);
\fill[color=white!50!black] (1,4) circle (0.5);
\fill[color=white!50!black] (2,4) circle (0.5);
\fill[color=white!50!black] (3,4) circle (0.5);
\fill[color=white!50!black] (4,4) circle (0.5);
\fill[color=white!50!black] (5,4) circle (0.5);
\draw (6,4) circle (0.5);
\fill[color=white!50!black] (0,3) circle (0.5);
\fill[color=white!50!black] (1,3) circle (0.5);
\fill[color=white!50!black] (2,3) circle (0.5);
\fill[color=white!50!black](3,3) circle (0.5);
\fill[color=white!50!black](4,3) circle (0.5);
\fill[color=white!50!black] (5,3) circle (0.5);
\fill[color=white!50!black] (6,3) circle (0.5);
\fill[color=white!50!black] (7,3) circle (0.5);
\fill[color=white!50!black](8,3) circle (0.5);
\fill[color=white!50!black] (9,3) circle (0.5);
\draw (10,3) circle (0.5);
\node at (1,5) {$\times$};
\node at (6,4) {$\times$};
\node at (10,3) {$\times$};
\node at (0,2) {4};
\node at (1,2) {$\varnothing$};
\node at (2,2) {3};
\node at (3,2) {3};
\node at (4,2) {3};
\node at (5,2) {3};
\node at (6,2) {$\varnothing$};
\node at (7,2) {2};
\node at (8,2) {2};
\node at (9,2) {2};
\node at (10,2) {$\varnothing$};
\node at ( 0,6) {\emph{1}};
\node at ( 1,6) {\emph{2}};
\node at ( 2,6) {\emph{3}};
\node at ( 3,6) {\emph{4}};
\node at ( 4,6) {\emph{5}};
\node at ( 5,6) {\emph{6}};
\node at ( 6,6) {\emph{7}};
\node at ( 7,6) {\emph{8}};
\node at ( 8,6) {\emph{9}};
\node at ( 9,6) {\emph{10}};
\node at (10,6) {\emph{11}};
\end{tikzpicture}}

\newtheorem{theorem}{Theorem}
\newtheorem{lemma}[theorem]{Lemma}
\newtheorem{proposition}[theorem]{Proposition}
\newtheorem{observation}[theorem]{Observation}
\theoremstyle{definition}

\newtheorem{example}{Example}

\allowdisplaybreaks

\begin{document}
\title{Counting prime juggling patterns}
\author{Esther Banaian\thanks{College of St.\ Benedict, Collegeville, MN 56321, USA {\tt embanaian@csbsju.edu}} \and
Steve Butler\thanks{Iowa State University, Ames, IA 50011, USA {\tt butler@iastate.edu}} \and
Christopher Cox\thanks{Carnegie Mellon University, Pittsburgh, PA 15213, USA {\tt cocox@andrew.cmu.edu}} \and
Jeffrey Davis\thanks{University of South Carolina, Columbia, SC 29208, USA {\tt davisj56@email.sc.edu}} \and
Jacob Landgraf\thanks{Michigan State University, East Lansing, MI 48824, USA {\tt landgr10@msu.edu}} \and
Scarlitte Ponce\thanks{California State University, Monterey Bay, Seaside, CA 93955, USA {\tt scponce@csumb.edu}}}
\date{\empty}
\maketitle

\begin{abstract}
Juggling patterns can be described by a closed walk in a (directed) state graph, where each vertex (or state) is a landing pattern for the balls and directed edges connect states that can occur consecutively.  The number of such patterns of length $n$ is well known, but a long-standing problem is to count the number of prime juggling patterns (those juggling patterns corresponding to cycles in the state graph).  For the case of $b=2$ balls we give an expression for the number of prime juggling patterns of length $n$ by establishing a connection with partitions of $n$ into distinct parts.  From this we show the number of two-ball prime juggling patterns of length $n$ is $\big(\gamma-o(1)\big)2^n$ where $\gamma=1.32963879259\ldots$.  For larger $b$ we show there are at least $b^{n-1}$ prime cycles of length $n$.
\end{abstract}

\section{Introduction}
Juggling has many interesting connections with combinatorics (see \cite{BuhlerGraham,Cards,Polster}).  There are several ways to describe juggling patterns, and each description gives some information about various properties of the patterns.  One of the most useful ways to describe a juggling pattern is with state graphs.

The \emph{state graph for $b$ balls} is an infinite directed graph where the vertices (or \emph{states}) correspond to a schedule of when balls will land and directed edges join states that can occur consecutively.  In particular, a state is a $0$-$1$ vector indexed by $\mathbb{N}$ where a $1$ in position $i$ indicates that a ball will land $i$ ``beats'' in the future (the number of $1$'s in the state vector equals the number of balls $b$).  Given a state the possible transitions are as follows:
\begin{itemize}
\item If the leading entry in the state is $0$, then there is no ball scheduled to land.  Moving forward one beat, we delete the first entry, and all other entries shift down by one.
\item If the leading entry in the state is $1$, then there is a ball scheduled to land.  Moving forward one beat, we delete the first entry, all other entries shift down by one, and then put a $1$ somewhere which is currently $0$. (That is, the ball goes into the hand, time moves forward one beat, and then we ``throw'' the ball so that it lands at a time in the future that does not already have a ball scheduled to land.)
\end{itemize}
In the language of state vectors, if $\mathbf{c}=\langle c_1,c_2,\ldots\rangle$ and $\mathbf{d}=\langle d_1,d_2,\ldots\rangle$ are states in the graph, then $\mathbf{c}\to\mathbf{d}$ if and only if $d_i\ge c_{i+1}$ for all $i\ge 1$.  A portion of the state graph when $b=2$ is shown in Figure~\ref{fig:state}.  Here we have only included those states where the largest index of a nonzero entry is at most four, and we have truncated the states.  We \emph{label} the edges to indicate the type of ``throw'' that occurred.  A ``$0$'' indicates no throw was made; otherwise, ``$i$'' indicates that we moved the leading $1$ into the $i$-th slot by an appropriate throw.

\begin{figure}[htf]
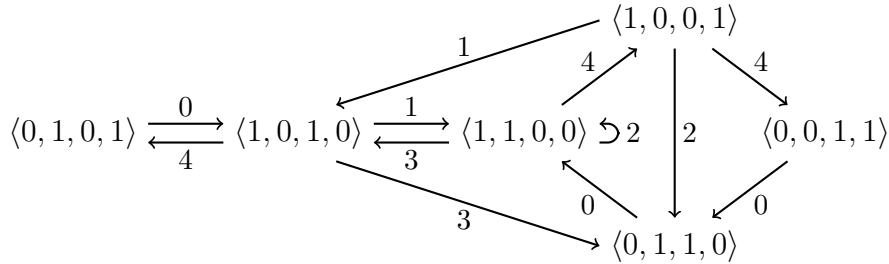
\label{stategraph}
\begin{center}
\picAAA
\caption{A subgraph of the two-ball state graph. The edge labels correspond to throw heights.}
\label{fig:state}
\end{center}
\end{figure}

Periodic juggling patterns can be found from the state graph by looking for closed walks. The length of the closed walk, which we denote by $n$, indicates the period of the corresponding pattern, i.e.\ the juggling pattern is formed by continuously repeating the walk so that the throw made at time $i$ (as indicated by the edge label) is the same as that made at time $i+n$.  The sequence of edge labels for the closed walk corresponds to the \emph{siteswap} sequence for the juggling pattern.  This is denoted by $t_1t_2\ldots t_n$ where each $t_i$ indicates the type of throw made at the $i$-th step.  Siteswap sequences are commonly used by jugglers when describing juggling patterns.  It is known that from the siteswap we can recover the closed walk in the state graph (see \cite{BuhlerGraham,Polster}).

Some of the closed walks, i.e.\ juggling patterns, from Figure~\ref{fig:state} are given in Table~\ref{table:state}.  Note that the first five listed have all of the $n$ states unique, i.e.\ the closed walks corresponds to cycles, while the last two have at least one state is repeated.  We call juggling patterns which correspond to cycles in the state graphs \emph{prime juggling patterns} (the notion of prime comes from noting every juggling pattern can be formed by appropriately combining prime juggling patterns).

\begin{table}[thb]
\centering
\[
\begin{array}{c|c}
\text{closed walk}&\text{siteswap}\\ \hline
\langle1,1,0,0\rangle{\to}\langle1,1,0,0\rangle&2\\
\langle1,0,1,0\rangle{\to}\langle0,1,0,1\rangle{\to}\langle1,0,1,0\rangle&40\\
\langle1,1,0,0\rangle{\to}\langle1,0,1,0\rangle{\to}\langle1,1,0,0\rangle&31\\
\langle1,1,0,0\rangle{\to}\langle1,0,0,1\rangle{\to}\langle0,0,1,1\rangle{\to}\langle0,1,1,0\rangle{\to}\langle1,1,0,0\rangle&4400\\
\langle1,1,0,0\rangle{\to}\langle1,0,0,1\rangle{\to}\langle1,0,1,0\rangle{\to}\langle0,1,1,0\rangle{\to}\langle1,1,0,0\rangle&4130\\
\langle1,1,0,0\rangle{\to}\langle1,0,1,0\rangle{\to}\langle0,1,0,1\rangle{\to}\langle1,0,1,0\rangle{\to}\langle1,1,0,0\rangle&3401\\
\langle0,1,0,1\rangle{\to}\langle1,0,1,0\rangle{\to}\langle0,1,1,0\rangle{\to}\langle1,1,0,0\rangle{\to}\langle1,0,1,0\rangle{\to}\langle0,1,0,1\rangle&03034
\end{array}
\]
\caption{Some of the juggling patterns from Figure~\ref{fig:state}.}
\label{table:state}
\end{table}

An exact expression for the number of juggling patterns with $b$ balls and period $n$ has been known for decades (see \cite{BuhlerGraham}), and is approximately $\big((b+1)^n-b^n\big)/n$.  The number of prime juggling patterns of period $n$ and $b$ balls, denoted $P(n,b)$, has never been determined for non-trivial values.  A small step towards counting prime juggling patterns was achieved by Chung and Graham \cite{ChungGraham} (see also \cite{ButlerGraham}) who were able to enumerate \emph{primitive} juggling patterns, i.e.\ patterns of period $n$ which do not repeat a fixed \emph{initial} state.  In this paper we will begin to address the enumeration problem of prime juggling patterns.

In Section~\ref{sec:parts} we will count the number of two-ball prime juggling patterns of period $n$ by showing a connection to the partitions of $n$ into distinct parts.  From this we will show in Section~\ref{sec:asym} that the number of two-ball prime juggling patterns of period $n$ is $\big(\gamma-o(1)\big)2^n$ for a known constant $\gamma=1.3296\ldots$.  In Section~\ref{sec:bigb} we give a lower bound of $b^{n-1}$ for the number of prime juggling patterns of length $n$ with $b$ balls.  Finally, we give some concluding remarks in Section~\ref{sec:conclusion}.

\section{Enumeration by ordered partitions}\label{sec:parts}
Our approach to counting will be to find a description of two-ball prime juggling patterns that is connected to ordered partitions.  The key observation we will use is that we can describe our patterns by how the spacings can occur in the states (i.e.\ the distance between the $1$'s in a state).

\subsection{Ternary sequences and spacings}
There is a bijection between period $n$, two-ball juggling patterns and ternary words of length $n$ using the letters $\{0,x,y\}$ with at least one occurrence of an $x$.  The letters are indicating the action on the $n$-th beat, and are interpreted as follows:
\begin{itemize}
\item $0$ indicates no ball was thrown at that beat.
\item $x$ indicates that a ball was thrown and that the next ball thrown will be the \emph{other} ball.
\item $y$ indicates that a ball was thrown and that the next ball thrown will be the \emph{same} ball.
\end{itemize}
In other words, a throw corresponding to an $x$ throws the ball so that it will land after the ball already in the air, while a throw corresponding to a $y$ throws the ball so that it will land before the ball already in the air.  We must have at least one $x$ since otherwise we will only throw at most one ball. (We note that this is related to the card interpretation of juggling patterns which we will visit in Section~\ref{sec:bigb}).

From the ternary word, we can reconstruct the sequence of throw heights, i.e.\ the siteswap.  (As noted earlier, once we have the siteswap sequence, we can find the closed walk in the state graph.)  Since a $y$ throw requires that the next ball thrown is the same ball, the throw height is the distance to the subsequent nonzero entry. On the other hand, an $x$ throw requires the next ball thrown to be different, so this ball will not land until the other ball is thrown by a throw corresponding to another $x$. Therefore, the throw height is the distance to the first nonzero entry following the subsequent $x$. These words are cyclic, so we wrap around at the ends when counting throw heights. 

\begin{example}\label{ternthrows}
We illustrate this process for $n = 11$ and the word $yx00x0y000x$. Below we have written the word in the first line, and in the second line we have indicated the throw heights (i.e.\ siteswap) corresponding to this word. 
\begin{center}
\begin{tabular}{*{11}{c}}
$y$&$x$&0&0&$x$&0&$y$&0&0&0&$x$\\
1&5&0&0&7&0&4&0&0&0&5\\
\end{tabular}
\end{center}
\end{example}

The next step is to determine when a ternary word corresponds to a prime juggling sequence.  We note that two-ball states are of the form $\langle \ldots,1,0,\ldots,0,1,\ldots \rangle$.  Since states with a leading $0$ will always result in a series of shifts until the first term is a $1$, it suffices to determine whether any states of the form
\[
\langle 1,\underbrace{0,\ldots,0}_{i},1 \rangle,
\]
are repeated, where $i$ is the number of $0$ entries between the two entries of $1$. For such a pattern, we will set the \emph{spacing} to be $i+1$. In other words, the spacing is the difference in the landing times for the two balls.  Thus, a two-ball juggling pattern is prime if and only if all the states $\langle 1,\ldots,1 \rangle$ visited in the pattern have unique spacings. Our next step is to determine the spacings from the word.

Given a ternary word corresponding to a juggling pattern, define an \emph{anchor point} as the first nonzero entry following an $x$. By our convention, an anchor point is always the result of an $x$ throw, and in particular is the \emph{second} ball in the state until the immediately preceding $x$ throw occurs. Therefore, to determine the spacing between the balls throughout the juggling pattern, we count the spacing between each nonzero entry and the \emph{subsequent} anchor point.

\begin{example}\label{ternspace}
For the ternary word in Example~\ref{ternthrows}, we place an anchor symbol (${*}$) above each anchor point, and below each nonzero entry give the spacing to the next anchor point.
\begin{center}
\begin{tabular}{*{11}{c}}
\scriptsize{${*}$}&&&&\scriptsize{${*}$}&&\scriptsize{${*}$}&&&&\\[-3pt]
$y$&$x$&0&0&$x$&0&$y$&0&0&0&$x$\\
$4$&$3$&&&$2$&&$5$&&&&$1$
\end{tabular}
\end{center}
Since these spaces are all distinct, the corresponding juggling pattern is prime.
\end{example}

For each anchor point in a word, we build a set of spacings connected to that anchor point. In Example~\ref{ternspace}, these sets are $\{4,3\}$, $\{2\}$, and $\{5,1\}$. Note that the sum of the distances between adjacent anchor points, which is also the sum of the largest entries in each set, is $n$. It is also possible to reconstruct our ternary word, and hence juggling pattern, given an ordered collection of sets of spacings.

\begin{lemma}\label{lem:reconstruct}
Given $S_1,S_2,\ldots,S_k$, where each $S_i$ is a nonempty set and the sum of the largest entries is $n$, then there is a unique cyclic ternary word of length $n$ so that the sets of spacings are $S_1,S_2,\ldots,S_k$ and for $1\le i\le k-1$, if $S_i$ is associated with a given anchor point, then $S_{i+1}$ is associated with the following anchor point.
\end{lemma}
\begin{proof}
We form the ternary word in reverse by applying the following algorithm.
\begin{itemize}
\item Start with a word of length $n+1$ where the first $n$ entries are $0$ and the last entry is an active point, temporarily marked ``?''.
\item For the sets $S_k,S_{k-1},\ldots,S_1$ (in that order) repeat the following action for the set $S_i$, 
\begin{itemize}
\item Suppose that $S_i=\{s_1,s_2,\ldots,s_t\}$ and that $s_1<s_2<\cdots<s_t$. In the entry precisely $s_1$ entries to the left of the active point, replace the $0$ with an ``$x$''.  For $j=s_2,\ldots,s_t$, in the entry precisely $j$ entries to the left of the active point, replace the $0$ with a ``$y$''.
\item Change the active point to be the furthest left nonzero entry (equivalently, the entry formed using element $s_t$ in $S_i$).
\end{itemize}
\item Delete the ``?'' in the last entry.
\end{itemize}

Since each set is nonempty and we place an $x$ immediately to the left of our active point, then the active points used in the construction will give the anchor points of the resulting ternary word.  Further, by the construction for each anchor point we will produce spacings that corresponded precisely to the set $S_i$ used in the placements, and since we worked backwards when constructing the sets of orderings, we will have that $S_i$ comes immediately before $S_{i+1}$.  Finally, we note that since the sum of the largest entries in the $S_i$ is $n$, then we will place an anchor point in the first entry, which by wraparound is the same as placing a point in entry $n+1$. This justifies deleting the ``?'' at the end.  (Note that an anchor point could be either an $x$ or $y$ which is why we did not initially specify the entry).
\end{proof}


\subsection{Ordered partitions}
As we have already noted, the largest entries of the sets of spacings form a partition of $n$.  If we restrict to considering prime juggling patterns, then the spacings are distinct, and the largest entries of the sets of spacings form a partition of $n$ into \emph{distinct} parts.

Starting with a partition of $n$ into distinct parts, we consider the number of ways to form sets of spacings that give a prime juggling pattern (i.e.\ number of ways to add possible additional elements while ensuring that there are no repeats).  

\begin{example}\label{Parts}
For the partition $2+7+11$ of $20$, we consider the number of ways to form a set of spacings that correspond to a prime juggling pattern with largest parts $2$, $7$, and $11$.  In particular, we consider the Ferrer's diagram for the partition.
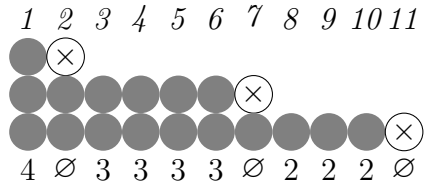
\begin{center}
\picQQQ
\end{center}
We have that the numbers $2$, $7$ and $11$ are the largest elements of the sets, and so it remains to determine what happens with the other values.  The number under each column indicates how many options we have for a particular value (i.e.\ we can include it in any subset whose largest part is greater than the value \emph{or} we can include it in none of the sets).  Since the choices for the values are independent, the number of possible sets of spacings is  $4\cdot3^4\cdot2^3$.
\end{example}

\begin{theorem}\label{thm:parts}
Let $P(n,2)$ be the number of two-ball prime juggling patterns of period $n$.  Then 
\[
P(n,2)=\sum_{t}\bigg(\sum_{\substack{p_1>\cdots>p_t \ge 1 \\ p_1+\cdots+p_t=n}}\frac{1}{t(t+1)}\prod_{i=1}^t\bigg(\frac{i+1}{i}\bigg)^{p_i}\bigg).
\]
\end{theorem}
\begin{proof}
We have already noted that a juggling pattern is prime if and only if no spacings are repeated.  We also know that given a set of spacings, together with a relative ordering of the sets, we can form a juggling pattern.  So it suffices to count the number of ways to form the sets of spacings with no repeated elements, keeping in mind that we must also keep track of the cyclic orderings of these sets.

In particular, given $n=p_1+p_2+\cdots+p_t$ with $p_1>p_2>\cdots>p_t\ge 1$, i.e.\ a partition of $n$ into distinct parts, the number of ways to form a set of spacings where the largest elements come from these parts is
\begin{equation}\label{eq:part}
\frac{(t+1)^{p_t}t^{p_{t-1}-p_{t}}(t-1)^{p_{t-2}-p_{t-1}}\cdots 2^{p_1-p_2}}{(t+1)t(t-1)\cdots 2}=
\frac{\displaystyle\prod_{i=1}^t\bigg(\frac{i+1}{i}\bigg)^{p_i}}{(t+1)!}.
\end{equation}
To see this, let $p_{t+1}=0$ and then for $2\le i\le t+1$ we have $p_{i-1}$ columns in the partition diagram that have at least $i-1$ elements, and similarly $p_i$ columns in the partition diagram that have at least $i$ elements.  In particular, there are exactly $p_{i-1}-p_i$ columns in the partition diagram that have exactly $i-1$ elements.  Each of these columns have $i$ different options for what happens to that value, except for the value $p_{i-1}$ which is already used as one of the largest elements in the sets of spacings.  Therefore these columns contribute $i^{p_{i-1}-p_i-1}$, giving \eqref{eq:part}.

We now also need to account for the cyclic orderings of the sets. Given $t$ sets, there are $(t-1)!$ such orderings. We can conclude that, for the given partition, we have
\[
\frac{1}{t(t+1)}\prod_{i=1}^t\bigg(\frac{i+1}{i}\bigg)^{p_i}
\]
different prime juggling patterns.  We now sum over all partitions with a fixed number of parts, and similarly sum over all possible number of parts to get the result.
\end{proof}

\section{Asymptotics}\label{sec:asym}
From Theorem~\ref{thm:parts} we can find the number of prime juggling patterns for any $n$.  In Table~\ref{table:2prime} we give these values for $n\le 30$.  By examining the data it appears that we are approximately doubling at each step.  In this section we will show that this reflects the behavior of these numbers.  In particular, we will establish the following.

\begin{theorem}\label{thm:asym}
We have $P(n,2) = \big(\gamma - o(1)\big)2^n$, where
\begin{equation*} 
\gamma = \frac{1}{2} + {1\over 2}\sum_{t\geq 2}\bigg(\prod_{i=2}^t{i-1\over 2^i-i-1}\bigg) = 1.3296387925905428331319\ldots
 \end{equation*}
\end{theorem}

\begin{table}[th]
\centering
\begin{tabular}{|r|r|}\hline
$n$&$P(n,2)$\\ \hline
 $1$&   $1$\\ \hline
 $2$&   $2$\\ \hline
 $3$&   $5$\\ \hline
 $4$&   $10$\\ \hline
 $5$&   $23$\\ \hline
 $6$&   $48$\\ \hline
 $7$&   $105$\\ \hline
 $8$&   $216$\\ \hline
 $9$&   $467$\\ \hline
$10$&   $958$\\ \hline
\end{tabular}
\hfil\hfil
\begin{tabular}{|r|r|}\hline
$n$&$P(n,2)$\\ \hline
$11$ & $2021$\\\hline
$12$ & $4146$\\\hline
$13$ & $8631$\\\hline
$14$ & $17604$\\\hline
$15$ & $36377$\\\hline
$16$ & $73876$\\\hline
$17$ & $151379$\\\hline
$18$ & $306822$\\\hline
$19$ & $625149$\\\hline
$20$ & $1263294$\\\hline
\end{tabular}
\hfil\hfil
\begin{tabular}{|r|r|}\hline
$n$&$P(n,2)$\\ \hline
$21$&  $2563895$\\ \hline
$22$&  $5169544$\\ \hline
$23$&  $10454105$\\ \hline
$24$&  $21046800$\\ \hline
$25$&  $42451179$\\ \hline
$26$&  $85334982$\\ \hline
$27$&  $171799853$\\ \hline
$28$&  $344952010$\\ \hline
$29$&  $693368423$\\ \hline
$30$&  $1391049900$\\ \hline
\end{tabular}
\caption{$P(n,2)$ for $1\le n\le 30$.}
\label{table:2prime}
\end{table}

\subsection{Upper bound}
If we let $c_t(n)$ be the number of ways to place spacings in the partitions of $n$ into exactly $t$ distinct parts, then by Theorem~\ref{thm:parts} we have
\begin{equation}\label{eq:fixedt}
c_t(n) = \sum_{\substack{p_1>\cdots>p_t \ge 1 \\ p_1+\cdots+p_t=n}}
\frac{1}{t(t+1)}\prod_{i=1}^t\bigg(\frac{i+1}{i}\bigg)^{p_i}.
\end{equation}

\begin{proposition}\label{prop:recur}
The values $c_t(n)$ satisfy the recurrence,
\[
c_t(n) = \sum_{k=1}^{\big\lfloor\frac{n-{t\choose2}}{t}\big\rfloor} (t-1)(t+1)^{k-1}c_{t-1}(n-kt) 
\]
\end{proposition}
\begin{proof}
Starting with \eqref{eq:fixedt} we have
\begin{align*}
c_t(n)&=\sum_{\substack{p_1>\cdots>p_t\ge1\\p_1+\cdots+p_t=n}}\frac1{t(t+1)}\prod_{i=1}^t\bigg(\frac{i+1}i\bigg)^{p_i}\\
&=\sum_{\substack{p_1>\cdots>p_t\ge1\\p_1+\cdots+p_t=n}}\frac{(t+1)^{p_t}}{t(t+1)}\prod_{i=1}^{t-1}\bigg(\frac{i+1}i\bigg)^{p_i-p_t}\\
&=\sum_{k=1}^{\big\lfloor\frac{n-{t\choose2}}{t}\big\rfloor}(t-1)(t+1)^{k-1}\sum_{\substack{p_1>\cdots>p_t=k\\p_1+\cdots+p_t=n}}\frac1{t(t-1)}\prod_{i=1}^{t-1}\bigg(\frac{i+1}i\bigg)^{p_i-k}\\
&=\sum_{k=1}^{\big\lfloor\frac{n-{t\choose2}}{t}\big\rfloor}(t-1)(t+1)^{k-1}\sum_{\substack{p_1'>\cdots>p_{t-1}'\ge1\\p_1'+\cdots+p_{t-1}'=n}}\frac1{t(t-1)}\prod_{i=1}^{t-1}\bigg(\frac{i+1}i\bigg)^{p_i'}\\
&=\sum_{k=1}^{\big\lfloor\frac{n-{t\choose2}}{t}\big\rfloor}(t-1)(t+1)^{k-1}c_{t-1}(n-kt).
\end{align*}
In going from the first to the second line we pull out $(t+1)^{p_t}$ (using that we have a telescoping product).  In going from the second line to the third line we group by the size of $p_t$, calling this parameter $k$, and note that since the parts are distinct we must have that $kt+{t\choose 2}\le n$ (i.e.\ our partition contains at least a $t\times k$ block and a triangle on the first $t-1$ entries).  In going from the third line to the fourth line we drop the size of each part by $k$ and now have a partition of $n-kt$ using $t-1$ parts.  Finally, in going from the fourth line to the fifth line we note that we now have the definition of $c_{t-1}(n-kt)$ inside of the sum.
\end{proof}

Since we have that $P(n,2)=\sum_tc_t(n)$, we can work on bounding the size of each $c_t(n)$.  This is achieved by the next result.

\begin{proposition}\label{prop:bound}
There exist constants $q_t$ so that $c_t(n) \le q_t2^n$ where $q_1 = \frac{1}{2}$ and for all $t \ge 2$,
\[
   q_t=\bigg({t-1\over 2^t-t-1}\bigg) q_{t-1}.
\]
\end{proposition}
\begin{proof}
Putting $t=1$ into \eqref{eq:fixedt} we have $c_1(n)=\frac122^n$, establishing $q_1=\frac12$.  Now assuming by induction we have established that $c_{t-1}(n)\le q_{t-1}2^n$, then by Proposition~\ref{prop:recur} we have
\begin{align*}
c_t(n)&=\sum_{k=1}^{\big\lfloor{n-{t\choose2}\over t}\big\rfloor}(t-1)(t+1)^{k-1}c_{t-1}(n-kt)
\le \sum_{k=1}^{\big\lfloor{n-{t\choose2}\over t}\big\rfloor}(t-1)(t+1)^{k-1}q_{t-1}2^{n-kt}\\
&\le \frac{(t-1)2^n}{t+1} q_{t-1}\sum_{k\ge1}\bigg({t+1\over 2^t}\bigg)^{k}
= \frac{(t-1)2^n}{t+1} q_{t-1}\frac{{t+1\over2^t}}{1-{t+1\over 2^t}}\\
&=\underbrace{\bigg(\frac{t-1}{2^t-t-1}\bigg) q_{t-1}}_{=q_t} 2^n.\qedhere
\end{align*}
\end{proof}

From the preceding proposition we can conclude for $t\ge2$ that
\[
q_t=\frac12\prod_{i=2}^t\frac{i-1}{2^i-i-1},
\]
in particular we have that $\gamma=\sum_{t\ge 1}q_t$.

We now have everything we need for the upper bound.

\begin{lemma}\label{lem:upper}
We have $P(n,2)\le \gamma 2^n$.
\end{lemma}
\begin{proof}
Using Proposition~\ref{prop:bound} we have
\begin{align*}
P(n,2)&=\sum_{t\ge 1}c_t(n)\le \sum_{t\ge 1}q_t2^n=\bigg(\frac12+\frac12\sum_{t\ge 2}\prod_{i=2}^t\frac{i-1}{2^i-i-1}\bigg)2^n=\gamma 2^n.\qedhere
\end{align*}
\end{proof}

\subsection{Lower Bound}
The key in establishing the upper bound is in Proposition~\ref{prop:recur} to give an upper bound on each individual $c_t(n)$ in Proposition~\ref{prop:bound}.  We will use a similar approach for establishing the lower bound.

\begin{proposition}\label{prop:lower}
There exist constants $q_t$ and $r_t$ so that $c_t(n) \ge q_t2^n-r_t{\sqrt3}^n$ where $q_1 = q_2= \frac{1}{2}$, $r_1=0$, $r_2=\frac{4\sqrt3}9$ and for all $t \ge 3$,
\[
   q_t=\bigg({t-1\over 2^t-t-1}\bigg) q_{t-1}\text{ and }
   r_t=\frac{t-1}{\sqrt3^t-t-1}r_{t-1}+2\bigg({2^t\over t+1}\bigg)^{(t-1)/2}q_{t-1}.
\]
\end{proposition}
\begin{proof}
Putting $t=1$ into \eqref{eq:fixedt} we have $c_1(n)=\frac122^n$, establishing $q_1=\frac12$ and $r_1=0$.  Now using Proposition~\ref{prop:recur} we have
\begin{align*}
c_2(n)&=\sum_{k=1}^{\lfloor{n-1\over2}\rfloor}3^{k-1}\frac122^{n-2k}
=\frac162^n\sum_{k=1}^{\lfloor{n-1\over2}\rfloor}\bigg(\frac34\bigg)^k
=\frac162^n\cdot\frac{\frac34-\big(\frac34\big)^{\lfloor{n-1\over2}\rfloor+1}}{1-\frac34}\\
&=\frac122^n-\frac232^n\bigg(\frac34\bigg)^{\lfloor{n-1\over2}\rfloor+1}
\ge\frac122^n-\frac232^n\bigg(\frac34\bigg)^{{n-1\over2}}
=\frac122^n-\frac{4\sqrt3}9\sqrt3^n,
\end{align*}
establishing $q_2=\frac12$ and $r_2=\frac{4\sqrt3}9$ (the inequality follows by noting that if we make the exponent smaller we are subtracting a larger value). Now assuming by induction we have established that $c_{t-1}(n)\ge q_{t-1}2^n-r_{t-1}{\sqrt3}^n$ we have
\begin{align*}
c_t(n)&=\sum_{k=1}^{\big\lfloor{n-{t\choose2}\over t}\big\rfloor}(t-1)(t+1)^{k-1}c_{t-1}(n-kt)\\
&\ge \sum_{k=1}^{\big\lfloor{n-{t\choose2}\over t}\big\rfloor}(t-1)(t+1)^{k-1}\big(q_{t-1}2^{n-kt}-r_{t-1}{\sqrt3}^{n-kt}\big)\\
&={(t-1)2^n\over t+1}q_{t-1}\sum_{k=1}^{\big\lfloor{n-{t\choose2}\over t}\big\rfloor}\bigg({t+1\over 2^t}\bigg)^k-{(t-1){\sqrt3}^n\over t+1}r_{t-1}\sum_{k=1}^{\big\lfloor{n-{t\choose2}\over t}\big\rfloor}\bigg({t+1\over {\sqrt3}^t}\bigg)^k.
\end{align*}

For the second term we have
\begin{align*}
{(t-1){\sqrt3}^n\over t+1}r_{t-1}\sum_{k=1}^{\big\lfloor{n-{t\choose2}\over t}\big\rfloor}\bigg({t+1\over {\sqrt3}^t}\bigg)^k
&\le {(t-1){\sqrt3}^n\over t+1}r_{t-1}\sum_{k\ge1}\bigg({t+1\over {\sqrt3}^t}\bigg)^k\\
&\le {(t-1){\sqrt3}^n\over t+1}r_{t-1}{{t+1\over\sqrt3^t}\over 1-{t+1\over\sqrt3^t}}\\
&={t-1\over\sqrt3^t-t-1}r_{t-1}\sqrt3^n.
\end{align*}

For the first term we first sum and split it into two parts, i.e.\ 
\begin{multline*}
{(t-1)2^n\over t+1}q_{t-1}\sum_{k=1}^{\big\lfloor{n-{t\choose2}\over t}\big\rfloor}\bigg({t+1\over 2^t}\bigg)^k=
{(t-1)2^n\over t+1}q_{t-1}\cdot\frac{\frac{t+1}{2^t}-\big(\frac{t+1}{2^t}\big)^{\big\lfloor{n-{t\choose2}\over t}\big\rfloor+1}}{1-\frac{t+1}{2^t}}\\
=\frac{t-1}{2^t-t-1}q_{t-1}2^n-\frac{(t-1)2^t}{(t+1)(2^t-t-1)}q_{t-1}\bigg(\frac{t+1}{2^t}\bigg)^{\big\lfloor{n-{t\choose2}\over t}\big\rfloor+1}2^n.
\end{multline*}
Proceeding similarly as before and using that $t\ge 3$ (so that among other things $(t+1)^{1/t}<\sqrt3$) we have
\begin{multline*}
\frac{(t-1)2^t}{(t+1)(2^t-t-1)}q_{t-1}\bigg(\frac{t+1}{2^t}\bigg)^{\big\lfloor{n-{t\choose2}\over t}\big\rfloor+1}2^n
\\=\underbrace{\frac{t-1}{t+1}}_{\le1}\underbrace{\frac{2^t}{2^t-t-1}}_{\le 2}q_{t-1}\bigg(\frac{t+1}{2^t}\bigg)^{\big\lfloor{n-{t\choose2}\over t}\big\rfloor+1}2^n
\le 2q_{t-1}\bigg(\frac{t+1}{2^t}\bigg)^{{n-{t\choose2}\over t}}2^n\\
= 2q_{t-1}\bigg(\frac{2^t}{t+1}\bigg)^{(t-1)/2}\big((t+1)^{1/t}\big)^n<2q_{t-1}\bigg(\frac{2^t}{t+1}\bigg)^{(t-1)/2}\sqrt3^n.
\end{multline*}

Putting this altogether we have
\[
c_t(n)\ge \underbrace{\frac{t-1}{2^t-t-1}q_{t-1}}_{=q_t}2^n-
\bigg(\underbrace{\frac{t-1}{\sqrt3^t-t-1}r_{t-1}+2\bigg({2^t\over t+1}\bigg)^{(t-1)/2}q_{t-1}}_{=r_t}\bigg)\sqrt3^n.\qedhere
\]
\end{proof}

\begin{lemma}\label{lem:lower}
Given any $\varepsilon>0$, for all $n$ sufficiently large $P(n,2)>(\gamma-\varepsilon)2^n$.
\end{lemma}
\begin{proof}
Since $\gamma=\sum_{t\ge 1}q_t$ and $q_t>0$, there is some $m$ so that $\sum_{t=1}^mq_t>\gamma-\frac12\varepsilon$.  For this $m$ we can use Proposition~\ref{prop:lower} to get
\begin{equation*}
P(n,2){=}\sum_{t\ge1}c_t(n){\ge}\sum_{t=1}^mc_t(n)
\ge \bigg(\sum_{t=1}^mq_t\bigg)2^n{-}\bigg(\underbrace{\sum_{t=1}^mr_t}_{=A}\bigg)\sqrt3^n{>}(\gamma-\frac12\varepsilon)2^n-A\sqrt3^n.
\end{equation*}
Since $\sqrt3<2$ then for $n$ sufficiently large we have $A\sqrt3^n\le \frac12\varepsilon2^n$.  In particular for such large $n$ we have $P(n,2)\ge (\gamma-\varepsilon)2^n$ establishing the result.
\end{proof}

Theorem~\ref{thm:asym} now follows immediately by combining Lemmas~\ref{lem:upper} and \ref{lem:lower}.

\section{Lower bound for prime patterns with $b\ge 2$ balls}\label{sec:bigb}
We can also consider the problem of counting prime juggling patterns for $b\ge3$ balls.  The natural thing is to again consider anchors, though it is unclear whether anchors should be one or more balls.  Each grouping of anchors should have some variation of the partition statistics similar to what was presented here, but a key difficulty is finding how to connect the anchors together.  In Table~\ref{data} we give some data for $P(n,b)$.  

\begin{table}[thb]
\centering
\[
\begin{array}{|l||r|r|r|}\hline
  P(n,b)  &    b{=}3&   b{=}4&   b{=}5\\ \hline\hline
 n{=}\phantom{1}1&      1&       1&       1\\ \hline
 n{=}\phantom{1}2&      3&       4&       5\\ \hline
 n{=}\phantom{1}3&     11&      19&      29\\ \hline
 n{=}\phantom{1}4&     36&      83&     157\\ \hline
 n{=}\phantom{1}5&    127&     391&     901\\ \hline
 n{=}\phantom{1}6&    405&    1663&    4822\\ \hline
 n{=}\phantom{1}7&   1409&    7739&   27447\\ \hline
 n{=}\phantom{1}8&   4561&   33812&  149393\\ \hline
 n{=}\phantom{1}9&  15559&  153575&  836527\\ \hline
n{=}10&  50294&   677901&   4610088\\ \hline
n{=}11& 169537&  3075879&  25846123\\ \hline
n{=}12& 551001& 13586581& 142296551\\ \hline
\end{array}
\]
\caption{Some data for $P(n,b)$.}
\label{data}
\end{table}

For $b=3,4,5$, the number of patterns appears to grow roughly as an exponential function in $b$.  We will establish a lower bound which supports this belief.

\begin{proposition}\label{prop:lb}
$P(n,b)\ge b^{n-1}$.
\end{proposition}

We will use the card interpretation of juggling sequences to help establish this result.  For a fixed $b$ there are $b+1$ cards, denoted $C_0$, $C_1$, \dots, $C_b$, which have on the left and right of each card $b$ slots (corresponding to the balls). These slots are connected by tracks which either go straight across (i.e.\ $C_0$) or have the bottom track drop down and then reposition itself relative to the other tracks. In particular, the cards keep track of the \emph{relative order} of the balls at any given time.  The set of these cards for $b=4$ is shown in Figure~\ref{fig:cards}.

\begin{figure}[hft]
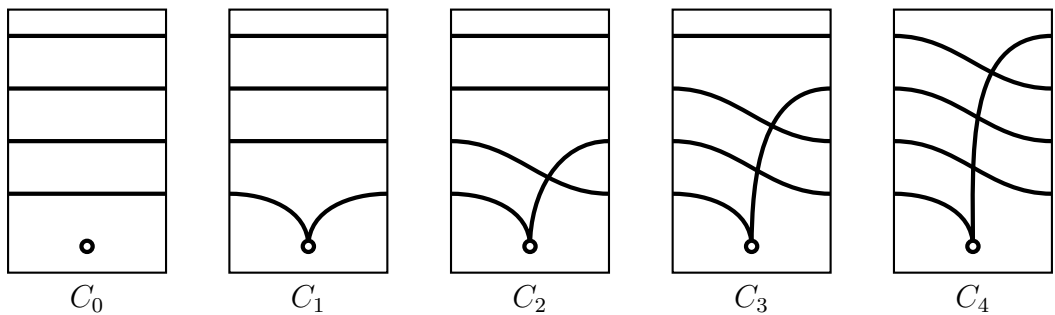

\centering
\picRRR
\caption{The set of juggling cards for $b=4$.}
\label{fig:cards}
\end{figure}

All siteswap patterns of length $n$ with $b$ balls correspond bijectively with $n$ cards drawn with replacement from $\{C_0,C_1,\ldots,C_b\}$ placed consecutively in some order with the card $C_b$ used at least once (see \cite{Cards,Polster}).  This can be seen by noting that placing the cards together gives a variant of the juggling diagram (i.e., the diagram that traces out the path of the balls over a time window of period $n$) from which the juggling pattern can be recovered.

\begin{proof}[Proof of Proposition~\ref{prop:lb}]
There are precisely $b^{n-1}$ ways to place $n$ cards consecutively where the first $n-1$ cards are drawn with replacement from $\{C_0,C_1,\ldots,C_{b-1}\}$ and the $n$-th card is $C_b$.  It suffices to show that these are prime (distinctness comes from the uniqueness of $C_b$).

To see that they are prime we note that the cards keep track of the ordering of the balls and because we only use $C_b$ in the last card it must be the case that the ball which is scheduled to land last among the thrown balls remains in  that ordering until the $n$-th step.  In terms of the states, this says that the last $1$ shifts down by $1$ at every step.  In particular, the location of the last $1$ will be different in all of the states, which forces the states to be distinct, i.e., a prime juggling pattern.
\end{proof}

\section{Concluding remarks}\label{sec:conclusion}
We have looked at what happens when we fix $b=2$ and let $n$ get large.  Alternatively one could fix $n$ and let $b$ get large. This variation was considered by Ron Graham \cite{Graham} who showed the following.
\begin{align*}
\text{For $b\ge 1$, }&P(2,b)=b\\
\text{For $b\ge 1$, }&P(3,b)=b^2+b-1\\
\text{For $b\ge 2$, }&P(4,b)=b^3+\frac32b^2-\frac12b-3\\
\text{For $b\ge 3$, }&P(5,b)=b^4+2b^3+2b^2+b-29\\
\text{For $b\ge 4$, }&P(6,b)=b^5+\frac52b^4+\frac{10}3b^3-4b^2-\frac{191}6b - 23
\end{align*}
In particular we note that when $n$ is fixed and $b$ gets large that \emph{most} juggling patterns are prime, while the $b=2$ case and data in Table~\ref{data} indicates that when $b$ is fixed and $n$ gets large that \emph{most} juggling patterns are not prime.

\bigskip

One common restraint for looking at the mathematics of juggling is to assume that we always throw a ball at each beat.  In terms of siteswap sequences, this means all of the $t_i\ge 1$.  One common approach to this is to first carry out computations with allowing $t_i\ge0$ and have one less ball, and then increase each $t_i$ by $1$ (i.e.\ adding $1$ to every throw height increases the number of balls by $1$).  We can do this for prime juggling patterns because of the following.

\begin{observation}
The siteswap sequence $t_1t_2\ldots t_n$ is prime if and only if the siteswap sequence $t_1't_2'\ldots t_n'$, where $t_i'=t_i+1$, is prime.
\end{observation}

This is a consequence of the following more general result.

\begin{theorem}\label{thm:isomgraph}
The state graph with $b$ balls is isomorphic to the (induced) subgraph of the state graph with $b+1$ balls consisting of vertices whose state starts with $1$.
\end{theorem}
\begin{proof}
The embedding works by sending $\langle a_1,a_2,\ldots \rangle$ in the the state graph for $b$ balls to $\langle 1,a_1,a_2,\ldots \rangle$ in the state graph for $b+1$ balls.

It remains to show  $\langle a_1,a_2,\ldots \rangle \stackrel{c}{\longrightarrow} \langle b_1,b_2,\ldots \rangle$ in the state graph for $b$ balls if and only if $\langle 1,a_1,a_2,\ldots \rangle \stackrel{c+1}{\longrightarrow} \langle 1,b_1,b_2,\ldots \rangle$. So suppose $\langle a_1,a_2,\ldots \rangle \stackrel{c}{\longrightarrow} \langle b_1,b_2,\ldots \rangle$ in the state graph for $b$ balls. We now break into two cases depending on $c$.
\begin{itemize}
\item $c=0$. In this case $a_1 = 0$ and $b_{k} = a_{k+1}$ for all $k$. It follows that the edge in the state graph with $b$ balls is of the form
\[
\langle0,a_2,\ldots\rangle\stackrel{0}{\longrightarrow}\langle a_2,a_3\ldots\rangle.
\]
While in the state graph for $b+1$ balls the edge is of the form
\[
\langle1,0,a_2,\ldots\rangle\stackrel{1}{\longrightarrow}\langle 1,a_2,a_3\ldots\rangle.
\]
\item $c\ge1$. In this case $a_1 = 1$, $a_{c+1} = 0$, $b_c = 1$, and $b_k = a_{k+1}$ for all $k \neq c$.  It follows that the edge in the state graph with $b$ balls is of the form
\[
\langle1,a_2,\ldots,a_{c},a_{c+1}=0,a_{c+2},\ldots\rangle\stackrel{c}{\longrightarrow}\langle a_2,\ldots,a_{c},b_c=1,a_{c+2},\ldots\rangle.
\]
While in the state graph for $b+1$ balls the edge is of the form
\[
\langle1,1,a_2,\ldots,a_{c},a_{c+1}=0,a_{c+2},\ldots\rangle\stackrel{c+1}{\longrightarrow}\langle 1,a_2,\ldots,a_{c},b_c=1,a_{c+2},\ldots\rangle.
\qedhere
\]
\end{itemize}
\end{proof}

Therefore, by Theorem~\ref{thm:isomgraph}, we have that the formula provided in Theorem~\ref{thm:parts} also counts the number of three-ball prime juggling patterns of period $n$ where a ball is thrown at each beat.

\bigskip

\noindent\textbf{Acknowledgments.}~~The authors are grateful for many useful comments and discussions with Ron Graham on the mathematics of juggling.  The research was conducted at the 2015 REU program held at Iowa State University which was supported by NSF DMS 1457443.

\end{document}